\documentclass[11pt]{article}

\usepackage{graphicx,amsmath,amssymb,amsthm,color}
\usepackage[all]{xy}
\usepackage[page]{appendix} 
\usepackage{hyperref}
\usepackage[section]{placeins}
\usepackage{float}

\newtheorem{thm}{Theorem}[section]
\newtheorem{lemma}[thm]{Lemma}

\input{xy}
\xyoption{all}

\begin{document}

\date{}
\title{A fake Klein bottle with bubble}
\author{W.H. Mannan}


\maketitle

\vspace{-9mm}
\begin{abstract}{We resolve the question of the existence of a finite 2-complex with the same fundamental group and Euler characteristic as a Klein bottle with a bubble, but homotopically distinct to it.}
\end{abstract}

\bigskip

{\tiny\noindent {\bf MSC classes} Primary: 57M05, 57K20.  Secondary: 20C07, 16S34, 20C10, 55P15, 55N25

}

\bigskip
{\tiny\noindent {\bf Acknowledgments} Thanks to Johnny Nicholson for productive discussions while attending``A Festival Remembering Victor Snaith", supported by The Heilbronn Institute, The London Mathematical Society and the Snaith family.  Thanks also to the referee of the published version, for pointing out some relevant and exciting recent work in geometric group theory.}

\section{Introduction}

Among the first homotopy invariants that students of topology are typically introduced to are the Euler characteristic and fundamental group.  Among the first classes of spaces such students may study (after finite graphs)  will often be finite 2-complexes.  As such it is not uncommon for the more adept students to question if these two invariants can distinguish between all 
homotopy classes of finite 2-complexes.

It has long been known that they cannot \cite[\S1.7]{Jens}.  Lustig \cite{Lust} produced infinitely many finite 2-complexes with the same Euler characteristic  and fundamental group, but pairwise distinct second homotopy modules.  Harlander and Jensen did this in the case where the fundamental group is the trefoil group \cite{Jens2}, building on the initial examples of Dunwoody and Berridge \cite{Berr,Dunw}.  See \cite{Nich} for more recent examples.  For finite fundamental groups, homotopically distinct finite 2-complexes with the same Euler characteristic  were found by Metzler for fundamental group ${C_5}^3$ \cite[p.105]{Metz, Lati1}.  Recently the present author and Popiel did this for fundamental group $Q_{28}$ \cite{Mann3}.

None of these spaces are easy to visualise.  Nor is it easy for a novice to see exactly what property distinguishes the spaces in question.  A space that {\it is} easy to visualise is a Klein bottle with a bubble $K^\circ$: formally, take a Klein bottle $K$, remove a small disk and replace it with a $2$-sphere, identifying the equator of the $2$-sphere with the boundary of the removed the disk.  Note $K^\circ$ is homotopy equivalent to the wedge $K\vee S^2$.  For aesthetic reasons, we prefer to work with $K^\circ$.

The finite $2$-complex $K^\circ$ has an easily visualised homotopy invariant property.  Its universal cover is the plane, with standard square tiling by translates of a fundamental cell, and a ``bubble" in each square (e.g. small disk removed and 2-sphere inserted). 
See Figure \ref{bubbly} below.

\begin{figure}[H]
\centering
\includegraphics[scale=.4]{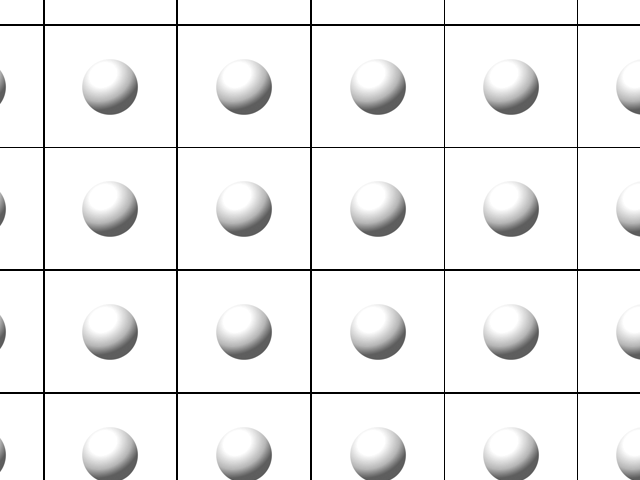}
\caption{The universal cover of $K^\circ$}
\label{bubbly}
\end{figure}

Any bubble can be translated to any other bubble by a unique element of the covering group.  Collectively, the bubbles generate the second homology group $H_2(\widetilde{K^\circ})$ over $\mathbb{Z}$.  Thus $H_2(\widetilde{K^\circ})$ is freely generated by a single element, as a module over $\mathbb{Z}[\pi_1(K^\circ)]$.

Motivated by the work of Harlander and his student Misseldine  \cite{Harl}, \cite[\S1.8]{Jens}, an open question in low dimensional topology for the last 14 years has been: Is there a finite 2-complex, $X$ with the same fundamental group and Euler characteristic as $K^\circ$, but with $H_2(\widetilde{X})$ not freely generated as a $\mathbb{Z}[\pi_1(X)]$ module?  We resolve this:

\bigskip
\noindent{\bf Theorem A.} {There exists a finite 2-complex $X$, with $\pi_1(X)\cong\pi_1(K^\circ)$ and $\chi(X)=\chi(K^\circ)$, but with $H_2(\widetilde{X})$ not freely generated as a $\mathbb{Z}[\pi_1(X)]$ module. }

\bigskip
Many questions regarding the existence of a finite 2-complex with specified (homotopy invariant) properties
arise from the discovery of a finite $D(2)$ complex with those properties (a $D(n)$ complex is one with no cohomology - twisted or otherwise - above dimension $n$).  In the present case Harlander and Misseldine discovered a finite $D(2)$ complex satisfying the conditions of Theorem A \cite{Harl}.

The motivation for these questions is Wall's $D(2)$ problem.  In 1965 Wall showed  that for $n>2$, every finite $D(n)$ complex is homotopy equivalent to a finite $n$-complex \cite{Wall}.  Swan and Stallings subsequently verified this when $n=1$ \cite{Stal, Swan1}.  The case $n=2$ remains one of the major open questions in low dimensional topology, known as Wall's $D(2)$ problem: Is there a finite $D(2)$ complex, which is not homotopy equivalent to a finite 2-complex?   If a finite $D(2)$-complex is found with homotopy invariant properties not possessed by any finite 2-complex, Wall's $D(2)$ problem would be solved.

It had been hoped that Theorem A would be false, (so $H_2(\widetilde{X})$ would always be free for a finite 2-complex, sharing fundamental group and Euler characteristic with $K^\circ$), meaning that Harlander and Misseldine's $D(2)$ complex would resolve the $D(2)$ problem.  The hope that $H_2(\widetilde{X})$ is always free came from the fact that unlike other one-relator groups (such as the trefoil group) the relation modules for the Klein bottle group are indeed free (see \cite[Theorem 1.15]{Jens} and \cite{Lars}).  Theorem A shows that this approach does not suffice.

More generally, given a surface $Y$ and integer $r$, one could ask if there is a finite 2-complex with the same fundamental group and Euler charecteristic as $Y\vee \vee_{i=1}^r S^2$, but with non-isomorphic second homotopy module.  If $Y$ is a torus then the Quillen-Suslin theorem tells us that this cannot happen, as the second homotopy module would be stably free, hence free. 

The finite $2$-complex $X$ of Theorem A is the presentation complex of a 2 generator, 2 relation presentation.  From Schanuel's lemma we may conclude that $H_2(\widetilde{X})\oplus  \mathbb{Z}[\pi_1(X)]\cong \mathbb{Z}[\pi_1(X)]^2$.  We thank the referee for drawing attention to the recent beautiful result, that for $Y$ a closed surface of genus greater than $e^{1000000}$, there are no non-free modules $N$, satisfying $N\oplus  \mathbb{Z}[\pi_1(Y)]\cong \mathbb{Z}[\pi_1(Y)]^2$ \cite[Corollary 4]{Avra}.  Thus for hyperbolic surfaces of high genus, there is no immediate analogue of our construction for the Klein bottle  \cite[Theorem 6]{Avra}.  The possibility of  stably free modules with larger minimal generating sets has not been ruled out.  See \cite{Avra2} for further developments in this direction.

\section{Presentations of the Klein bottle group}\label{pressec}

A finite presentation of a group $G$ consists of a set of generators $g_1,\cdots,g_n$, and a set of relators $R_1,\cdots,R_m$: words in the generators that represent the trivial element of the group.  The $g_i$ must generate $G$, and the $R_j$ must normally generate the subgroup of relators in the free group on the $g_i$.  In this case we write:
\begin{eqnarray*}
\mathcal{P}= \langle g_1,\cdots,g_n\,\vert\,R_1,\cdots,R_m\rangle
\end{eqnarray*} 
is a presentation for $G$.

The Cayley complex $X_\mathcal{P}$ of a finite presentation $\mathcal{P}$ is the union of a wedge (over a point) of directed loops, indexed by the generators of $\mathcal{P}$, with a set of disks indexed by the relators of $\mathcal{P}$, where the boundary of the disk indexed by relator $R_j$ is identified with the composition of loops which $R_j$ expresses (reading right to left).

By Van Kampen's theorem, if $\mathcal{P}$ presents a group $G$, then $\pi_1(X_\mathcal{P})=G$.  Now let:
\begin{eqnarray*}
\mathcal{P}= \langle x,y\,\vert\,y^{-1}xyx\rangle
\end{eqnarray*}
We have $X_\mathcal{P}$ is the Klein bottle $K$ (see Figure \ref{Klein_Cayley}), so $\mathcal{P}$ presents the fundamental group of the Klein bottle $\pi_1(K)$.

\begin{figure}[H]
\centering
\includegraphics[scale=.35]{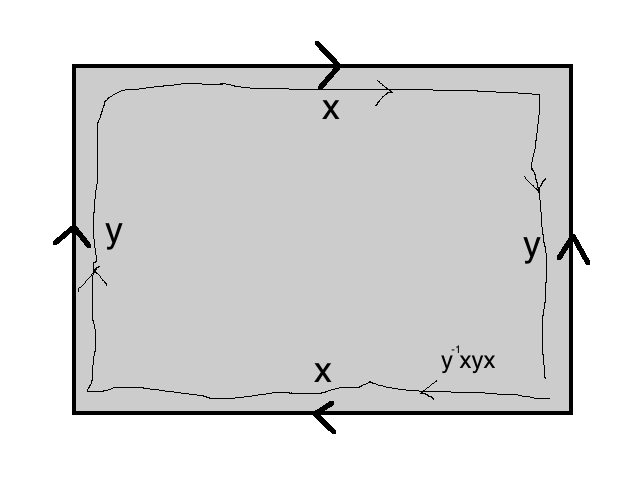}
\caption{The Klein bottle as a Cayley complex}
\label{Klein_Cayley}
\end{figure}

In this section we will show that:
\begin{eqnarray*}
\mathcal{Q}= \langle x,y\,\vert\,\,y^{-2}xy^2x^{-1},\,\, x^{-3}(y^{-1}xy)x^2(y^{-1}x^{-2}y) \rangle
\end{eqnarray*}
also presents $\pi_1(K)$.  By Van Kampen's theorem, we know that attaching a disk to a finite complex $X$, by identifying the boundary of the disk with a contractible loop in $X$, does not alter the fundamental group of $X$.  Thus we will be able to conclude that: $$\pi_1(X_\mathcal{Q})\cong\pi_1(K)\cong\pi_1(K^\circ).$$

Also note that $\chi(X_\mathcal{Q})$ is the number of 2-cells (disks) in $\mathcal{Q}$ minus the number of 1-cells  (loops) plus the number of 0-cells (points) which is \newline $2-2+1=1$.  Performing the same calculation for $K=X_\mathcal{P}$ gives $\chi(K)=1-2+1=0$.  Finally attaching a single 2-cell to get $K^\circ$ we conclude $\chi(K^\circ)=\chi(K)+1=1$, so:$$\chi(X_\mathcal{Q})=1=\chi(K^\circ).$$

That is, once we have shown that $\mathcal{P}$ and $\mathcal{Q}$ present the same group, we will be able to conclude that $X_\mathcal{Q}$ and $K^\circ$ have the same fundamental group and Euler characteristic.  In the next section we will show that $H_2(\widetilde{X_\mathcal{Q}})$ is not freely generated by a single element over $\mathbb{Z}[\pi_1(K)]$, thus proving Theorem A.

For the rest of this section we focus on showing that $\mathcal{Q}$ presents the same group as $\mathcal{P}$.  At first sight, the relators in  $\mathcal{Q}$ look long and complicated.  However the first relator merely tells us that $y^2$ commutes with $x$.  This is clearly true in the group presented by $\mathcal{P}$ (conjugating $x$ by $y$ twice  takes $x\mapsto x^{-1}\mapsto x$).  Let $a=y^{-1}xy, b=x$.  The second relation of $\mathcal{Q}$ then merely states that $b^{-3}ab^2a^{-2}$ is trivial, or equivalently: \begin{eqnarray}\label{abrel1}
ab^2=b^3a^2.\end{eqnarray}
Again this holds in the group presented by $\mathcal{P}$, as there we have $ab=1$.  To show that $\mathcal{P}$ and $\mathcal{Q}$ present the same group, it remains only to show that $ab=1$ holds in the group presented by $\mathcal{Q}$.  We know by definition that $y^{-1}by=a$ holds in this group.  From the first relator we then know that $$y^{-1}ay=y^{-2}xy^{2}=x=b.$$
Thus conjugating both sides of (\ref{abrel1}) by $y$, we merely interchange the $a$'s and $b$'s, to get:\begin{eqnarray}\label{abrel2} 
ba^2=a^3b^2.\end{eqnarray}

To complete our proof that $ab=1$ we use (\ref{abrel1}) and (\ref{abrel2}) and invoke  \cite[Lemma 3.1]{Mann3}.  For completeness, we repeat the proof here.

\begin{lemma}\cite[Lemma 3.1]{Mann3} \label{abr=3}
For a group $G$, let $a,b \in G$ satisfy (\ref{abrel1}), (\ref{abrel2}). Then $ba=1$.
\end{lemma}

\begin{proof}
Multiplying (\ref{abrel1}) through by $a^2$ on the left we get: $$a^2b^3a^2=a^3b^2=ba^2,$$
from (\ref{abrel2}).  Thus $a^2b^2=1$ so $b^2a^2=1$ and (\ref{abrel2}) reduces to $b^{-1}=a$.
\end{proof}

Thus $\mathcal{P}$ and $\mathcal{Q}$ present the same group and we may conclude:

\begin{lemma} \label{sameinv}
We have $\chi(X_\mathcal{Q})=\chi(K^\circ)$ and  $\pi_1(X_\mathcal{Q}))\cong\pi_1(K^\circ).$
\end{lemma}

\section {Non-freeness of $H_2(\widetilde{X_\mathcal{Q}})$ }

Consider the cellular decomposition induced on the universal cover $\widetilde{X_\mathcal{P}}$.  This is a standard square tiling of the plane (see Figure \ref{Square_Tiling}).

\begin{figure}[H]
\centering
\includegraphics[scale=.7]{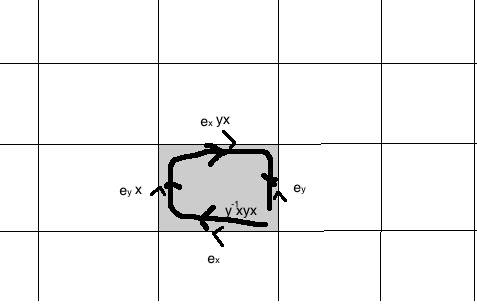}
\caption{Cellular decomposition lifted to universal cover of Klein bottle}
\label{Square_Tiling}
\end{figure}

As the plane is contractible, the resulting cellular chain complex gives us an exact sequence of (right) modules over $\mathbb{Z}[\pi_1(K)]$:
$$
0\to C_2(\widetilde{X_\mathcal{P}})\stackrel {d_2} \to C_1(\widetilde{X_\mathcal{P}})\stackrel {d_1} \to C_0(\widetilde{X_\mathcal{P}})\stackrel {} \to\mathbb{Z}\to 0
$$

Here $C_2(\widetilde{X_\mathcal{P}})$ is freely generated by a 2-cell $D$ indexed by the relator $y^{-1}xyx$.  Then $d_2$ is injective and exactness at $C_1(\widetilde{X_\mathcal{P}})$ tells us that $d_2(D)$ freely generates the kernel of $d_1$.

Now consider the cellular decomposition induced on the universal cover $\widetilde{X_\mathcal{Q}}$.  As $\widetilde{X_\mathcal{Q}}$ is simply connected, from its cellular chain complex we obtain the following exact sequence of (right) modules over $\mathbb{Z}[\pi_1(K)]$:

\begin{eqnarray}\label{Qchains}
0 \to H_2(\widetilde{X_\mathcal{Q}})\to C_2(\widetilde{X_\mathcal{Q}})\stackrel {d_2'} \to C_1(\widetilde{X_\mathcal{Q}})\stackrel {d_1'} \to C_0(\widetilde{X_\mathcal{Q}})\stackrel {} \to\mathbb{Z}\to 0
\end{eqnarray}

As $\mathcal{P}$ and $\mathcal{Q}$ have the same generating set, we can identify (\ref{Qchains})with:
\begin{eqnarray} \label{Pchains}
0 \to H_2(\widetilde{X_\mathcal{Q}})\to C_2(\widetilde{X_\mathcal{Q}})\stackrel {d_2'} \to C_1(\widetilde{X_\mathcal{P}})\stackrel {d_1} \to C_0(\widetilde{X_\mathcal{P}})\stackrel {} \to\mathbb{Z}\to 0
\end{eqnarray}

Here  $C_2(\widetilde{X_\mathcal{Q}})$ is freely generated over $\mathbb{Z}[\pi_1(K)]$ by 2-cells $D_1, D_2$ indexed by the relators of $\mathcal{Q}$.

In the previous section we noted that the relators of $\mathcal{Q}$ are trivial in the group presented by $\mathcal{P}$, where  $y^{-1}xyx$ is the only relator.  In fact we can express the relators of $\mathcal{Q}$ as products of conjugates of $y^{-1}xyx$.  

\begin{lemma} \label{prodconj}The following identities hold in the free group on $x,y$:

\begin{eqnarray*}
y^{-2}xy^2x^{-1}&=&(y^{-1}(y^{-1}xyx)y)(x(y^{-1}xyx)^{-1}x^{-1}),
\\
x^{-3}y^{-1}xyx^2y^{-1}x^{-2}y&=&(x^{-3}(y^{-1}xyx)x^3)(x^{-1}(y^{-1}xyx)^{-1}x)(y^{-1}xyx)^{-1}.
\end{eqnarray*}
\end{lemma}

\begin{proof}
Multiply out the right hand sides, and they will reduce to the left hand sides.
\end{proof}

Conjugating a relator by a word representing the element $g$ of the presented group corresponds to translating the boundary of its associated 2-cell by $g$ (see Figure \ref{Conj_Action}).  Multiplying two relators results in a relator whose associated 2-cell has boundary the sum of the boundaries of the 2-cells associated to the two relators.

\begin{figure}[H]
\centering
\includegraphics[scale=.564]{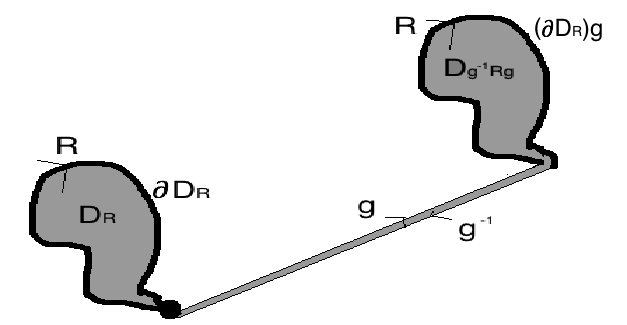}
\caption{Conjugation action on relator / boundary of associated 2-cell}
\label{Conj_Action}
\end{figure}

Thus from lemma \ref{prodconj} we conclude:\begin{eqnarray*}
d_2'(D_1)&=&d_2(D)(y-x^{-1})
\\
d_2'(D_2)&=&d_2(D)(x^3-x-1).
\end{eqnarray*}

\noindent We noted earlier that $d_2(D)$ freely generates the kernel of $d_1$.  Thus: 

\begin{lemma} \label{H2isker}
The module $H_2(\widetilde{X_\mathcal{Q}})$ is precisely the kernel of the module homomorphism  $\psi\colon \mathbb{Z}[\pi_1(K)]^2\to \mathbb{Z}[\pi_1(K)]$, mapping $$\left(\begin{array}{c} u \\ v\end{array}\right) \mapsto  (y-x^{-1})u+ (x^3-x-1) v.$$  
 \end{lemma} 

Also, from the exactness of (\ref{Pchains}) we know that $\psi$ is surjective. Explicitly we may write: $$
\psi\left(\begin{array}{c} -(x^4-x^2-x)y^{-1} \\ (x^{-3}-x^{-4}-y^{-1})\end{array}\right)=1.
$$

\noindent It remains to show that  $H_2(\widetilde{X_\mathcal{Q}})=\text{ker}\,\psi$ is not a free module over $\mathbb{Z}[\pi_1(K)]$.  In fact from lemma  \ref{H2isker} we know that: \begin{eqnarray}H_2(\widetilde{X_\mathcal{Q}})\oplus \mathbb{Z}[\pi_1(K)]\cong \mathbb{Z}[\pi_1(K)]^2 \label{stabfreeeqn}\end{eqnarray}
Thus if $H_2(\widetilde{X_\mathcal{Q}})$ were free, it must be generated freely by a single element.  Otherwise we would derive a contradiction of ranks, by tensoring both sides of (\ref{stabfreeeqn}) with $\mathbb{Z}$ over the ring homomorphism $\epsilon\colon \mathbb{Z}[\pi_1(K)]\to \mathbb{Z}$, mapping $g\mapsto 1$ for all $g\in \pi_1(K)$.

Let $S=\mathbb{Z}[\pi_1(K)]$.  Let $R=\mathbb{Z}[x,x^{-1}]\subset S$ be the ring of (finite) Laurent polynomials in $x$.  Then $R$ is a Noetherian domain and $S$ is the (skew) polynomial ring over $R$ with indeterminate $y$, where $y^{-1}xy=x^{-1}$.  

For any $\alpha(\neq 0)\in S$ we say the {\it degree} of $\alpha$, deg$(\alpha)$, is the greatest exponent of $y$ in $\alpha$ (with non-zero coefficient) minus the smallest.  Clearly ${\rm deg}(\alpha\beta)={\rm deg}(\alpha)+{\rm deg}(\beta)$. Thus there are no (left or right) zero divisors in $S$.

For any $\alpha\in S$ we say $\alpha$ is {\it monic} precisely when the largest exponent of $y$ with non-zero coefficient in $\alpha$ is $1\in R$.

\begin{lemma} \label{isotoideal}
We have an isomorphism of $S$-modules: $$H_2(\widetilde{X_\mathcal{Q}})\stackrel{\sim}\to \{v\in S\vert\,\, (x^3-x-1)v\in (y-x^{-1})S\}=V,$$
mapping $\left(\begin{array}{c} u \\ v\end{array}\right) \mapsto v$.
\end{lemma}

\begin{proof}. Given $\left(\begin{array}{c} u \\ v\end{array}\right)\in\text{ker}\,\phi$, we know $ (y-x^{-1})u+ (x^3-x-1) v=0$, so $(x^3-x-1) v=-(y-x^{-1})u$ and $v\in V$, so the map is well defined.  Conversely given $v\in V$, by construction we have $u\in S$ with $(x^3-x-1) v=(y-x^{-1})u$, so $\left(\begin{array}{c} -u \\ v\end{array}\right)\in \text{ker}\,\phi$, and the map is surjective.  

The map is injective, as $(y-x^{-1})$ is not a zero-divisor in $S$.
\end{proof}

Note that $V\lhd S$ is a right ideal.  As both $y^2$ and elements of the form $r(y^{-1}ry)$ for $r\in R$ are central in $S$, we have: \begin{eqnarray*}
(x^3-x-1)(y^2-1)&=&(y^2-1)(x^3-x-1)\\&=&(y-x^{-1})(y+x)(x^3-x-1)\\ {\rm and} \hspace{60mm}&\,\,&\hspace{60mm}\\
(x^3-x-1)&\,&(x^{-3}-x^{-1}-1)(y-x^{-1})=\\(y-x^{-1})&\,&(x^3-x-1)(x^{-3}-x^{-1}-1),
\end{eqnarray*}
so \begin{eqnarray}y^2-1, (x^{-3}-x^{-1}-1)(y-x^{-1})\in V. \label{elementsin}\end{eqnarray}

\begin{thm} \label{notfreemod}
The module $H_2(\widetilde{X_\mathcal{Q}})$ is not free over $\mathbb{Z}[\pi_1(K)]$.
\end{thm} \label{notfree}
\begin{proof} From lemma \ref{isotoideal} and (\ref{stabfreeeqn}) it suffices to show that $V$ is not a principal right ideal. From (\ref{elementsin}) we know ithat $V$ contains both a monic element and a degree $1$ element.  Thus if it were principal, any generator of $V$ would divide both elements (on the left), so is itself degree $1$ and can be chosen monic (by construction $V$ cannot contain degree $0$ elements) .  Thus we may write the generator $y+a$, with $a\in R$ and we have:
$$
(x^3-x-1)(y+a)\in (y-x^{-1})S,
$$
so $$
(x^3-x-1)(y+a)= (y-x^{-1})(x^{-3}-x^{-1}-1),
$$
equating coefficients on $y$.  Thus we would have $x^3-x-1$ dividing $x^3+x^2-1$ in the ring $\mathbb{Z}[x,x^{-1}]$.  However then, $x^3-x-1$  would also divide the difference between the cubics: $x^2+x$.  This is impossible, as $x^2+x$ has shorter length (difference between highest and lowest exponents on $x$) than $x^3-x-1$.
\end{proof}

\noindent Theorem A then follows from theorem \ref{notfreemod} and lemma \ref{sameinv}.

\bigskip
Note our argument is a special case of a general proof by Stafford \cite[Theorem 1.2 (Case (ii), $\delta=0$)]{Staff}.

%

\noindent W.H. Mannan,  \hfill {\it email}: \textit{wajid@mannan.info}

\noindent 125 Russell Road,

\noindent London,

\noindent SW19 1LN

\end{document}